\documentclass{conm-p-l}

\usepackage{amssymb,amsmath,euscript}

\usepackage{}

\newtheorem{theorem}{Theorem}[section]
\newtheorem{lemma}[theorem]{Lemma}
\newtheorem{proposition}[theorem]{Proposition}
\newtheorem{conjecture}[theorem]{Conjecture}

\theoremstyle{definition}

\theoremstyle{remark}

\numberwithin{equation}{section}
\def\Fq{{\mathbb F}_q}
\def\FF{{\mathbb F}}

\def\Fqmn{{\mathbb F}_{q^{mn}}}
\def\Fqtn{{\mathbb F}_{q^{2n}}}
\def\Fqn{{\mathbb F}_{q^{n}}}

\newcommand{\GL}{\mathrm{GL}}

\newcommand{\M}{\mathrm{M}}

\newcommand{\GCD}{\mathrm{GCD}}

\newcommand{\bav}{{\EuScript B}^{\alpha}_{(v_1,v_2)}}

\newcommand{\bavm}{{\EuScript B}^{\alpha}_{(v_1,\dots, v_m)}}

\newcommand{\Sb}{{\EuScript S}_{\beta}}
\begin{document}

\title[Enumeration of Splitting Subspaces]{Enumeration of Splitting Subspaces over Finite Fields}


\author{Sudhir R. Ghorpade}
\address{Department of Mathematics, 
Indian Institute of Technology Bombay,\newline \indent
Powai, Mumbai 400076, India.}
\email{srg@math.iitb.ac.in}

\author{Samrith Ram}
\address{Department of Mathematics,
Indian Institute of Technology Bombay,\newline \indent
Powai, Mumbai 400076, India.}
\email{samrith@gmail.com}


\date{}

\begin{abstract}
We 
discuss an elementary, yet unsolved, problem of Niederreiter concerning the enumeration of 
a class of subspaces of finite dimensional vector spaces over finite fields.  
A short and self-contained account of some
recent progress on this problem is included 
and 
some 
related problems are discussed. 
\end{abstract}

\maketitle


\section{Introduction}
\label{sec:in}
Finite fields have a remarkable property that finite dimensional vector spaces over them are naturally endowed with a canonical and compatible field structure. Indeed, we can simply ``move the $d$'' so as to write $\Fq^d \simeq \FF_{q^d}$, where $d$ is any positive integer and as usual, $\Fq$ denotes the finite field with $q$ elements. This leads to some interesting notions where the field structure and the linear structure are intertwined. One such notion is that of a splitting subspace, which appears to go back at least to Niederreiter (1995) in connection with 
his work on pseudorandom number generation. 
Here is the definition:

Let $m,n$ be positive integers, $q$ a prime power, and $\alpha \in \Fqmn$. 
An $m$-dimensional $\Fq$-linear subspace $W$ of $\Fqmn$ is said to be \emph{$\alpha$-splitting} if 
\begin{displaymath}
\Fqmn = W \oplus \alpha W \oplus \cdots \oplus \alpha^{n-1}W.
\end{displaymath}
Concerning these, Niederreiter \cite{N2} asked the following: 
given 
$\alpha\in \Fqmn$ such that $\Fqmn=\Fq(\alpha)$, what is the number of 
$m$-dimensional $\alpha$-splitting subspaces of $\Fqmn$?

Actually, the above question is a slightly more general version of the original question stated as an open problem in \cite[p. 11]{N2} where it is assumed that $q$ is prime 
and $\alpha$ is a primitive element  of $\Fqmn$ in the sense that it is a generator of the cyclic group 
$\Fqmn^*$ of nonzero elements of $\Fqmn$. 
But this general version seems quite natural and we will always consider Niederreiter's question in this setting. 

The main aim of this article is to make Niederreiter's question better known and to facilitate further research on it. We were motivated by the fact that till recently there had not been any significant progress on this question since it was posed more than 15 years ago. Recent progress came about primarily by relating this question to seemingly different questions in cryptography. This brought to the fore exciting connections not only with cryptography but also matrix theory and finite projective geometry via the so called block companion Singer cycles. As a result, a quantitative formulation of Niederreiter's question suggested itself and a small breakthrough was obtained in the form of a solution in the case of splitting planes, i.e., when $m=2$. We refer to Appendix \ref{app} and to \cite{Zeng, GSM, GR, L} for these developments. While the myriad connections are no doubt interesting, we wish to underline the fact that Niederreiter's question is a beautiful problem that is easy to state and 
is of interest in itself. With this in view, we give here an 
account of the recent progress on this question by focusing mainly on splitting subspaces \emph{per se} and relegating its connections to cryptography and such to an appendix at the end. 
In particular, we include a short and self-contained proof of the solution to 
Niederreiter's question in the case of splitting planes. 
Our original proof (cf. \cite{GR}) in the case $m=2$ used a result of Benjamin and Bennett \cite{BB}, which in turn was motivated by a question of Corteel, Savage, Wilf, and Zeilberger \cite{CSWZ} (see \cite[Rem. 4.2]{GGR} for more historical information). Here we have 
removed the dependence on Benjamin and Bennett \cite{BB} by 
means of an auxiliary result (Lemma \ref{GenBB}) 
and given a quick and independent proof of it by modifying an argument in \cite[Thm. 4.1]{GGR}. 
We have also used this opportunity to include certain 
variants of 
Niederreiter's question and 
some preliminary results concerning them. 
Finally, for the convenience of the reader (and at the suggestion of a referee), we include a brief appendix  where interconnections with cryptography, Singer cycles, etc., have been outlined. 

\section{Easy cases and guesses}
\label{sec1}

Fix, throughout this paper, positive integers $m,n$ and a prime power $q$.
Let us first note 
that for an arbitrary $\alpha\in \Fqmn$, there may not be any $\alpha$-splitting subspace; for example, if $\alpha\in \Fq$, then $\alpha^iW=W$ for every $m$-dimensional subspace $W$ and every $i\ge 0$, and so $W$ cannot be $\alpha$-splitting if $n>1$. 
To avoid such situations, we will always assume that $\alpha\in \Fqmn$ satisfies $\Fqmn=\Fq(\alpha)$. In this case,  
$\left\{1, \alpha, \alpha^2, \dots , \alpha^{mn-1}\right\}$ forms a $\Fq$-basis of $\Fqmn$ and hence 
$\left\{1, \alpha^n, \alpha^{2n}, \dots , \alpha^{(m-1)n}\right\}$ spans an  $m$-dimensional $\alpha$-splitting subspace of $\Fqmn$, say $\Lambda$.  
Let us define
$$
S(\alpha, m,n;q) := \text{the number of $\alpha$-splitting subspaces of $\Fqmn$ of dimension $m$}. 
$$
Niederreiter's question is to determine (a nice formula for) $S(\alpha, m,n;q)$. 
The case when $m$ or $n$ is equal to $1$ is quite trivial. Indeed, if $n=1$, then   
the only $m$-dimensional subspace, viz., $W=\Fqmn$, is $\alpha$-splitting for every  $\alpha\in \Fqmn$. 
On the other hand, if $m=1$ and if $\alpha\in \Fqmn=\FF_{q^n}$ is such that $\Fqmn=\Fq(\alpha)$, then every $1$-dimensional subspace is 
$\alpha$-splitting. Thus 
$$
S(\alpha, m,n;q) = \frac{q^{mn}-1}{q^m-1} \quad \text{ if } \min\{m, n\}=1.
$$
In fact, the fraction on the right is always a lower bound for $S(\alpha, m,n;q)$. To see this, it suffices to note 
two things: (i) if $\Lambda$ is as above,
then $\beta\Lambda$ is an $m$-dimensional $\alpha$-splitting subspace 
for every $\beta \in \Fqmn^*$, since  
$\alpha^j\beta = \beta\alpha^j$ for $0\le j\le n-1$, 
and (ii) $\Lambda = \Fq(\alpha^n)$ and $\Lambda^*$ is a subgroup of the cyclic group $\Fqmn^*$ of index $(q^{mn}-1)/(q^m-1)$ so that if we let 
$\beta \in \Fqmn^*$ vary over representatives of distinct cosets of $\Lambda^{*}$ in $\Fqmn^*$, then the corresponding subspaces 
$\beta\Lambda$ are distinct (in fact, essentially disjoint). \linebreak 
To work out a slightly 
nontrivial example, let us suppose $m=2$ and $n=2$. 
Recall that for any integers $a,b$ with $a\ge b\ge 0$, the number of $b$-dimensional subspaces of an $a$-dimensional vector space over $\Fq$ is given by the Gaussian binomial coefficient
$$
{\genfrac{[}{]}{0pt}{}{a}{b}}_q 
:= \frac{(q^{a}-1)(q^{a}-q)\cdots (q^{a}-q^{{b}-1})}{(q^{b}-1)(q^{b}-q)\cdots (q^{b}-q^{{b}-1})}.
$$
Let $W$ be subspace of $\FF_{q^4}$. Note that $W = \alpha W$ if and only if $W=0$ or $W=\FF_{q^4}$ (indeed, if $0\ne x\in W = \alpha W$, then 
$W$ contains the linearly independent elements $x, \, \alpha x, \, \alpha^2 x, \, \alpha^3 x$ and so $W=\FF_{q^4}$).
Now suppose $\dim W=2$ and $W$ is not $\alpha$-splitting. 
Then $L=W\cap \alpha W$ is a $1$-dimensional and $W = L + \alpha^{-1}L$. Conversely, if $L$ is a $1$-dimensional subspace of $\FF_{q^4}$, then 
$L + \alpha^{-1}L$ is a $2$-dimensional subspace of $\FF_{q^4}$ that is not $\alpha$-splitting. It follows that 
$$
S(\alpha, 2,4;q) = 
{\genfrac{[}{]}{0pt}{}{4}{2}}_q - {\genfrac{[}{]}{0pt}{}{2}{1}}_q 
= \frac{q^{4}-1}{q^2-1} \, q^2. 
$$
We now take an inspired leap and propose the following quantitative formulation of Niederreiter's question.

\medskip

{\bf Splitting Subspace Conjecture:} Let $\alpha\in \Fqmn$ satisfy 
$\Fqmn=\Fq(\alpha)$. Then
$$
S(\alpha, m,n;q) = \frac{q^{mn}-1}{q^m-1} \, q^{m(m-1)(n-1)}.
$$

\medskip

To be sure, the conjectural formula fits well with the 
examples considered above as well as the general lower bound for 
$S(\alpha, m,n;q)$. Still to arrive at it based only on a few examples is indeed quite a leap. As alluded to in the Introduction and explained in the appendix, the true inspiration, in fact, comes from a recent conjecture of Zeng, Han and He \cite{Zeng} and the subsequent work in \cite{GSM} and \cite{GR}. But at any rate, we have a nice specific problem, which seems to be open, in general. Its solution in the only nontrivial case known so far 
will be considered next. 

\section{Splitting planes}
\label{sec2}

We begin with a simple but useful observation 
that goes back to Niederreiter \cite[Lem. 3]{N2} and says the enumeration of splitting subspaces of $\Fqmn$ is equivalent to the enumeration of certain ordered bases of $\Fqmn$. To make this more precise, let us introduce some notation. 

Given any 
$\alpha, v_1, \dots, v_m\in \Fqmn$, we let 
$$
\bavm:=\left\{v_1, \dots , v_m, \, \alpha v_1, \dots , \alpha v_m, \, \dots , \,  \alpha^{n-1} v_1,\dots , \alpha^{n-1} v_m\right\}, 
$$  
with the understanding that $\bavm$ is to be regarded as an ordered set with $mn$ elements. In case $\bavm$ is an ordered basis of $\Fqmn$,
the set $\{v_1, \dots , v_m\}$ is necessarily a $\Fq$-basis of an $m$-dimensional subspace of $\Fqmn$ and we will refer to $\bavm$ as an 
\emph{$\alpha$-splitting ordered basis} of $\Fqmn$. The number of $\alpha$-splitting ordered bases of $\Fqmn$ will be denoted by
$N(\alpha, m,n;q)$. 
%

\begin{lemma}
\label{splitandbases}
Let $\alpha\in \Fqmn$ satisfy 
$\Fqmn = \Fq(\alpha)$, and let $v_1, \dots , v_m\in \Fqmn$. 
Then $\bavm$ is an ordered basis of $\Fqmn$ if and only if 
$\{v_1, \dots , v_m\}$ span an $m$-dimensional $\alpha$-splitting  subspace of $\Fqmn$. Consequently, 
$$
S(\alpha, m,n;q) = \frac{N(\alpha, m,n;q)}{\left|\GL_m(\Fq)\right|}, \quad \text{that is,} \quad N(\alpha, m,n;q) = S(\alpha, m,n;q) \prod_{i=0}^{m-1}(q^m-q^i).
$$
\end{lemma}

\begin{proof}
The first assertion is obvious. The second follows from the first by noting that the number of distinct ordered bases of an $m$-dimensional vector space over $\Fq$ is 
$\left|\GL_m(\Fq)\right| = \prod_{i=0}^{m-1}(q^m-q^i)$. 
\end{proof}

From now on, we will focus on the case of splitting planes, i.e., the case $m=2$.

\begin{lemma}
\label{nobases}
Let $\alpha\in \Fqtn$ be such that $\Fqtn = \Fq(\alpha)$. 
Then 
$$
N(\alpha, 2,n;q) = \left(q^{2n} - \nu -1\right)(q^{2n} - 1), 
$$
where 
$\nu$ denotes the cardinality of the set
$\Sigma$ of pairs $(f_1,f_2)$ of nonzero polynomials in $\Fq[X]$ of degree $< n$ 
with $f_2$ monic and $f_1, f_2$ 
relatively prime. 
\end{lemma}

\begin{proof}
Fix 
$v_1\in \Fqtn$ with $v_1\ne 0$. Then for any $v_2\in \Fqtn$, the ordered set 
$$
\bav =\left\{v_1, \, v_2, \, \alpha v_1, \, \alpha v_2, \, \dots , \,  \alpha^{n-1} v_1,\, \alpha^{n-1} v_2\right\} 
$$ 
is a $\Fq$-basis of $\Fqtn$ if and only if the ordered set
$$
\Sb := \left\{1, \, \beta, \, \alpha , \, \alpha \beta, \, \dots , \,  \alpha^{n-1} ,\, \alpha^{n-1} \beta\right\}
$$  
is linearly independent over $\Fq$, where $\beta: = v_2/v_1$. Now, $1, \alpha, \dots ,\alpha^{2n-1}$ are linearly independent over $\Fq$ and in particular, so are $1, \alpha, \dots ,\alpha^{n-1}$. Thus for any $\beta \in \Fqtn^*$, the ordered set $\Sb$ is $\Fq$-independent if and only if $\beta$ cannot be expressed as 
$$
\frac{a_0+ a_1 \alpha + \cdots + a_{n-1} \alpha^{n-1}}{b_0+ b_1 \alpha + \cdots + b_{n-1} \alpha^{n-1}}
$$
for some $a_i,b_i\in \Fq$ such that not all $a_i$ are zero and not all $b_i$ are zero ($0\le i\le n-1$). It follows that 
$\big\{\beta \in \Fqtn^* : \Sb \text{ is linearly independent}\big\} =  \Fqtn^* \setminus \Sigma_{\alpha}$, where 
$$
\Sigma_{\alpha} : = \left\{\frac{p_1(\alpha)}{p_2(\alpha)} : p_i\in \Fq[X]^*\text{ with } \deg (p_i) < n \text{ for } i=1,2 \right\}.
$$
Now 
consider 
$$
\Sigma : = \left\{\frac{f_1}{f_2} : f_i\in \Fq[X]^*\text{ with } \deg (f_i) < n \text{ for } i=1,2 \text{ and } \GCD(f_1, f_2) = 1\right\}.
$$
The map $\Sigma \to \Sigma_{\alpha}$ given by $(f_1,f_2)\mapsto f_1(\alpha)/f_2(\alpha)$ is 
clearly well-defined and surjective. Moreover, if $(f_1,f_2), (g_1,g_2)\in \Sigma$ are such that $f_1(\alpha)g_2(\alpha)= g_1(\alpha)f_2(\alpha)$, then 
$f_1g_2=g_1f_2$ because the minimal polynomial of $\alpha$ over $\Fq$ has degree $2n$. Further since $\GCD(f_1, f_2) = 1= \GCD(g_1,g_2)$ and 
since $f_2, g_2$ are monic, it follows that $f_2=g_2$ and therefore $f_1=g_1$. Thus $\Sigma_{\alpha}$ is in bijection with $\Sigma$, 
and hence 
upon letting $\nu = |\Sigma|$, we find 
$$
\left|\left\{\beta \in \Fqtn^* : \Sb \text{ is linearly independent }\right\}\right| = (q^{2n}-1 - \nu). 
$$
Finally, if we vary $v_1$ over the $(q^{2n}-1)$ elements of $\Fqtn^*$, then we readily see that the number of ordered bases of the form $\bav$ is 
equal to $\left(q^{2n} - \nu -1\right)(q^{2n} - 1)$. 
\end{proof}

The cardinality $\nu$ of the set $\Sigma $ appearing in Lemma \ref{nobases} will be determined using the following more general result concerning 
pairs of relatively prime polynomials. 

\begin{lemma}
\label{GenBB}
Let $N_1, N_2$ be positive integers with 
$N_1\ge N_2$ and let $\nu(N_1, N_2)$ denote the number of ordered pairs $(f_1, f_2)$ 
of coprime nonzero polynomials in $\Fq[X]$ with $f_2$  monic 
and $\deg f_i < N_i$ for $i=1,2$. 
Then 
$\nu(N_1, N_2) = q^{N_1+N_2-1} - 1. $
\end{lemma}

\begin{proof}
We can partition the set of  ordered pairs $(f_1, f_2)$ 
of nonzero polynomials in $\Fq[X]$ with $f_2$  monic 
and $\deg f_i < N_i$ for $i=1,2$  into disjoint subsets
$S_d$ ($0\le d < N_2$), 
where $S_d$ consists of 
pairs whose GCD is of degree $d$.
Given any monic polynomial $h\in \Fq[X]$ of degree $d$ and any coprime pair $\bigl(g_1, g_2\bigr)$ of nonzero polynomials with $g_2$ monic 
and $\deg g_i < N_i - d$ for $i = 1, 2$, it is easy to see that $\left(hg_1,  hg_2\right)\in S_d$. 
Conversely, if $\bigl(f_1,  f_2\bigr) \in S_d$, then the polynomial
$h = {\rm GCD}\bigl(f_1, f_2\bigr)$ is monic of degree $d$ and $(f_1/h, f_2/h)$
is a 
coprime pair comprising of a nonzero polynomial of degree $N_1-d$ and a monic polynomial of degree $N_2-d$.  
This shows that $\left|S_d\right| = q^{d} \nu(N_1-d, N_2-d)$ for $0\le d < N_2$. Since for any positive integer $N$, there are $(q^N-1)$ nonzero polynomials in $\Fq[X]$ 
of degree $<N$ and of these exactly $(q^N-1)/(q-1)$ are monic, it follows that  
\begin{equation} \label{Eq1}
\left(q^{N_1} -1 \right) \frac{\left(q^{N_2} -1\right)}{q-1} = \sum_{0\le d< N_2} \left|S_d \right|= \sum_{0\le d< N_2} q^{d} \nu\left(N_1-d, N_2-d\right).
\end{equation}
If $N_2 = 1$, we immediately obtain $\nu (N_1, N_2) = q^{N_1}-1$. On the other hand, if $N_2 > 1$, then substituting $N_i-1$ for $N_i$ ($i=1, 2$) in the above relation yields
\begin{equation}\label{Eq2}
\left(q^{N_1-1} -1 \right) \frac{\left(q^{N_2-1} -1\right)}{q-1} = \sum_{1\le d< N_2} q^{d-1} \nu\left(N_1-d, N_2-d\right).
\end{equation}
Multiplying equation \eqref{Eq2} by $q$ and subtracting the result from \eqref{Eq1}, and then making an elementary calculation, we see that 
$\nu\left(N_1,N_2\right) =q^{N_1+N_2-1} - 1$. 
\end{proof}

It is now a simple matter to show that the Splitting Subspace Conjecture holds in the affirmative when $m=2$ (and $n$ is arbitrary). 

\begin{theorem}
\label{mtwoconj}
Let $\alpha\in \Fqtn$ be such that $\Fqtn = \Fq(\alpha)$. Then 
$$
S(\alpha, 2,n;q) = \displaystyle \frac{q^{2n}-1}{q^2-1} \, q^{2(n-1)} .
$$

\end{theorem}

\begin{proof}
Follows from Lemmas \ref{splitandbases}, \ref{nobases}, and  \ref{mtwoconj}.
\end{proof}

\section{Refinements and Extensions} 
\label{sec:asymp}


For $\alpha\in \Fqmn$, let $\mathfrak{S}_{\alpha}$ denote the set of 
all $m$-dimensional $\alpha$-splitting subspaces of $\Fqmn$. 
By a \emph{pointed $\alpha$-splitting subspace} of dimension $m$ we shall mean a pair $(W,x)$ where $W \in \mathfrak{S}_{\alpha}$ 
and $x\in W$. The element $x$ may be referred to as the \emph{base point} of $(W,x)$. Given any $x\in \Fqmn$, we let
$
\mathfrak{S}_{\alpha}^x:=\left\{W\in \mathfrak{S}_{\alpha}: x\in W\right\}.
$

\begin{proposition}
\label{elemsplitprop}
Let $\alpha\in \Fqmn$ be such that $\Fqmn = \Fq(\alpha)$. Then 
$$
\left|\mathfrak{S}_{\alpha}^x\right| =\left|\mathfrak{S}_{\alpha}^y\right|  \quad \text{for any $x,y\in \Fqmn^*$.}
$$
Moreover, for any $x\in \Fqmn^*$, the set $\mathfrak{S}_{\alpha}^x$ is nonempty and 
$$
S(\alpha, m,n;q) 
= \left|\mathfrak{S}_{\alpha}^x\right|  \frac{q^{mn}-1}{q^m-1} . 
$$
\end{proposition}

\begin{proof}
If $x,y\in \Fqmn^*$ and $\beta=y/x$, then $W\mapsto \beta W$ gives a bijection of $ \mathfrak{S}_{\alpha}^x$ onto $ \mathfrak{S}_{\alpha}^y$.
Moreover, for any $x\in \Fqmn^*$, the $\Fq$-linear span of $\left\{x\alpha^{in}: 0\le i < m\right\}$
is clearly in $\mathfrak{S}_{\alpha}^x$ and thus
$\mathfrak{S}_{\alpha}^x$ is nonempty. Finally,  
by counting in two different ways the set $\left\{(W,x): W\in  \mathfrak{S}_{\alpha} \text{ and } x\in W\right\}$ of all pointed $\alpha$-splitting subspaces, we find $\left|\mathfrak{S}_{\alpha}\right| (q^m-1) = \left|\mathfrak{S}_{\alpha}^x\right|{(q^{mn}-1)}$, as desired. 
\end{proof}

It may be remarked that the lower bound for $S(\alpha, m,n;q)$ discussed in Section~\ref{sec1} is an immediate consequence of Proposition~\ref{elemsplitprop}.   
In light of Proposition \ref{elemsplitprop}, we 
see that the Splitting Subspace Conjecture is equivalent to the following simpler looking conjecture. 

\begin{conjecture}[Pointed Splitting Subspace Conjecture]
\label{PSSC}
Let $\alpha\in \Fqmn$ be such that $\Fqmn = \Fq(\alpha)$ and let $x\in \Fqmn^*$. Then the number of $m$-dimensional pointed $\alpha$-splitting subspaces  of $\Fqmn$ with base point $x$ is equal to $q^{m(m-1)(n-1)}$. 
\end{conjecture}

We remark that $q^{m(m-1)}$ is the number of nilpotent $m\times m$ matrices over $\Fq$, thanks to an old result of Fine and Herstein \cite{FH}. 
Thus a particularly nice way to prove 
the Pointed Splitting Subspace Conjecture could be to set up a natural bijection between $\mathfrak{S}_{\alpha}^x$ and the set of $(n-1)$-tuples (or if one prefers, pointed $n$-tuples) of nilpotent $m\times m$ matrices over $\Fq$. 

If the Splitting Subspace Conjecture were to hold in the affirmative, then an obvious consequence would be that the number of
$m$-dimensional pointed $\alpha$-splitting subspaces  of $\Fqmn$ is independent of the choice of $\alpha\in \Fqmn$ as long as it satisfies 
$\Fqmn = \Fq(\alpha)$. In other words, for any $\alpha, \beta \in \Fqmn$, 
$$
S(\alpha, m,n;q) = S(\beta, m,n;q) \quad \text{provided} \quad \Fqmn  = \Fq(\alpha)  = \Fq(\beta).
$$
In general we do not know if this weaker statement is true. The following result summarizes the cases where the answer is known. 

\begin{proposition}
\label{WeakSSC}
Let $\alpha \in \Fqmn$ be such that $\Fqmn  = \Fq(\alpha)$. If
$$
\beta = \frac{a\alpha^{q^r} +b}{c\alpha^{q^r} +d} \quad \text{for some nonnegative integer $r$ and } \begin{pmatrix} a & b \\ c & d \end{pmatrix} \in \GL_2(\Fq),
$$
then $\Fqmn = \Fq(\beta)$ and $S(\alpha, m,n;q) = S(\beta, m,n;q)$.
\end{proposition}

\begin{proof}
First, note that if $\beta = c\alpha$ for some $c\in \Fq^*$ or $\beta = \alpha +d$ for some $d\in \Fq$, then
$\Fqmn = \Fq(\beta)$ and in view of Lemma \ref{splitandbases}, we see that $\mathfrak{S}_{\alpha} = \mathfrak{S}_{\beta}$ 
and so $S(\alpha, m,n;q) = S(\beta, m,n;q)$. Further, if $\alpha \ne 0$ (which is necessarily the case if $mn>1$), then 
$\Fqmn = \Fq(1/\alpha)$ and again in view of Lemma \ref{splitandbases} and the fact that multiplication by the nonzero element $\alpha^{-(n-1)}$ preserves linear independence, it follows that 
$\mathfrak{S}_{1/\alpha} = \mathfrak{S}_{\alpha}$ and so $S(\alpha, m,n;q) = S(1/\alpha, m,n;q)$. Finally, note that for any nonnegative integer $r$, the elements 
$\alpha$ and $\alpha^{q^r}$ are Galois conjugate, i.e., they have the same minimal polynomial over $\Fq$, and therefore 
$\Fqmn = \Fq\left(\alpha^{q^r}\right)$ and there is a one-to-one correspondence between $\alpha$-splitting and $\alpha^{q^r}$-splitting subspaces,  induced by the corresponding element of the Galois group of $\Fqmn$ over $\Fq$. Combining these, we obtain the desired result. 
\end{proof}

It may be remarked that in view of Proposition \ref{WeakSSC} and the normal basis theorem \cite[p. 60]{LN}, we see that there is a $\Fq$-basis ${\EuScript B}$ 
of $\Fqmn$ such that each element of ${\EuScript B}$ generates $\Fqmn$ over $\Fq$ and $S(\alpha, m,n;q) = S(\beta, m,n;q)$ for all $\alpha, \beta \in {\EuScript B}$. 

Finally, we note that Niederreiter's question can also be posed in a more general situation where instead of considering multiples of an
$m$-dimensional subspace by powers of $\alpha$, we consider its transforms by an endomorphism of $\Fqmn$. More precisely, given any
$\Fq$-linear endomorphism $T: \Fqmn \to \Fqmn$, we say that an $m$-dimensional subspace $W$ of $\Fqmn$ is \emph{$T$-splitting} if 
$$
\Fqmn = W \oplus T(W) \oplus T^2(W) \oplus \cdots \oplus T^{n-1}(W),
$$
where $T^j$ denotes the $j$-fold composite of $T$ with itself ($0\le j < n$). We let 
$$
S_T(m,n;q) = \text{the number of $m$-dimensional $T$-splitting subspaces of } \Fqmn.
$$
Evidently, if $T$ is the $\Fq$-linear endomorphism of $\Fqmn$ given by $x\mapsto \alpha x $, then $S_T(m,n;q) = S(\alpha,m,n;q)$. A more general variant of 
Niederreiter's question is to determine $S_T(m,n;q)$ for every $\Fq$-linear endomorphism $T$ of $\Fqmn$. An answer to this question does not seem to be known, even conjecturally. It should be noted, however, that certain restrictions on the structure of $T$ will be needed in order that 
$S_T(m,n;q)$ is nonzero and independent of the choice of $T$ in a suitable class. For example, if $m=1$, then the existence of $m$-dimensional $T$-splitting subspaces of $\Fqmn$ evidently forces $T$ to be cyclic and the minimal polynomial 
of $T$ to be the characteristic polynomial of $T$. 
A complete answer to the above variant of Niederreiter's question in this case is given below. 
On the other hand, if $n=1$, then  $W=\Fqmn$ is obviously the only $m$-dimensional $T$-splitting subspace, for any $T: \Fqmn \to \Fqmn$ and thus 
$S_T(m,1;q)=1$. 
 
\begin{proposition}
\label{EndoSSC}
Let $T: \Fqn \to \Fqn$ be a cyclic $\Fq$-linear endomorphism and 
let $p_T \in \Fq[X]$ be the minimal polynomial of $T$. Suppose 
$p_T=f_1^{e_1}\cdots f_k^{e_k}$
is the factorization of $p_T$ into positive powers of distinct monic irreducible polynomials $f_i \in \Fq[X]$ with $\deg(f_i)=n_i$ for $i=1, \dots , k$. Then 
$$
S_T(1,n;q)=
\frac{q^n}{q-1}  \prod_{i=1}^k\left(1-\frac{1}{q^{n_i}}\right). 
$$
\end{proposition}

\begin{proof}
Clearly any $1$-dimensional $T$-splitting subspace $W$ of $\Fqn$ is spanned by a cyclic vector for $T$. So 
it suffices to count the number of cyclic vectors for $T$. Let $v \in \Fqn$ be a cyclic vector for $T$. Then any other cyclic vector $w\in \Fqn$ of $T$ is necessarily of the form $f(T)v$ where $f\in \Fq[X]$ is 
such that $\deg f\leq n-1$. Hence $w$ is cyclic only if the $T$-annihilator of $w$ is precisely $p_T$. Now the annihilator of $f(T)v$ is $p_T$ if and only if $\gcd(f,p_T)=1$. Thus the number of $f \in \Fq[X]$ with $\deg f\leq n-1$ such that $w=f(T)v$ is cyclic is equal to the number of polynomials in  $\Fq[X]$ of degree $\leq n-1$ that are coprime to $p_T$. This is given by the $q$-analogue of the Euler totient function (cf. \cite[p. 122]{LN}) evaluated at the minimal polynomial of $T$, namely, 
$$
\Phi_q(p_T)=q^n \prod_{i=1}^k\left(1-\frac{1}{q^{n_i}}\right).
$$
Note that $f_1(T)v\neq f_2(T)v$  for distinct polynomials $f_1,f_2$ of degree at most $n-1$,  for otherwise $p_T \mid (f_1-f_2)$, which is a contradiction. Thus there are $\Phi_q(p_T)$ distinct cyclic vectors for $T$. Since each $1$-dimensional $T$-splitting subspace of $\Fqn$ is spanned by precisely $q-1$ distinct cyclic vectors, it follows that the number of $1$-dimensional $T$-splitting subspaces of $\Fqn$ is $\Phi_q\left(p_T\right)/(q-1)$, as desired. 
\end{proof}

It may be noted that if $\alpha \in \Fqmn$ is such that $\Fqmn  = \Fq(\alpha)$ and if $T$ is the $\Fq$-linear endomorphism of $\Fqmn$ given by $x\mapsto \alpha x $, then the minimal polynomial $p_T$ of $T$ is precisely the minimal polynomial of $\alpha$ over $\Fq$. Hence $p_T$ is irreducible and therefore 
the formula in Proposition \ref{EndoSSC} reduces to
$S(\alpha,1,n;q) =(q^n-1)/(q-1)$, 
exactly as observed in the beginning of Section \ref{sec2}.

\appendix

\section{Vector Recurrences and Singer Cycles}
\label{app}
 
As before, we fix  positive integers $m,n$ and a prime power $q$.
For any positive integer $d$, we denote, as usual, by $\M_d(\Fq)$ 
the set of all $d\times d$ matrices with entries in $\Fq$, and by $\GL_d(\Fq)$ 
the group of all nonsingular matrices in $\M_d(\Fq)$. 

Let $C_0, C_1, \dots, C_{n-1} \in M_m(\mathbb F_q)$. 
Given any \emph{initial state} in $(\Fq^m)^n$, i.e., an $n$-tuple $({\mathbf{s}}_0, \dots, {\mathbf{s}}_{n-1})$  of 
(row) vectors in  $\mathbb F_q^m$, 
the vector recurrence (of order $n$ over $\mathbb F_q^m$)
\begin{eqnarray}
{\mathbf{s}}_{i+n}={\mathbf{s}}_iC_0+{\mathbf{s}}_{i+1}C_1+\cdots +{\mathbf{s}}_{i+n-1}C_{n-1} \quad \mbox{for} \quad i=0,1,\dots \label{sigmalfsr} 
\end{eqnarray}
generates an infinite sequence ${\mathbf{s}}^{\infty}=({\mathbf{s}}_0, {\mathbf{s}}_1,\dots)$ of vectors in $\Fq^m$. It is easy to 
see that there are integers $r, n_0$ with $1\le r \le q^{mn}-1$ and $n_0\ge 0$ such that ${\mathbf{s}}_{j+r} = {\mathbf{s}}_j$ for all $j\ge n_0$.  
The least positive integer $r$ with this property is called the \emph{period} of ${\mathbf{s}}^{\infty}$ and the corresponding least nonnegative integer $n_0$ is called the \emph{preperiod} of ${\mathbf{s}}^{\infty}$. 
The sequence  ${\mathbf{s}}^{\infty}$ is said to be \emph{periodic} if its preperiod is $0$. 
The vector recurrence \eqref{sigmalfsr} is said to be \emph{primitive} if for any choice of nonzero initial state, the infinite 
sequence generated by it is periodic of period  $q^{mn}-1$. Vector recurrences are also known as word oriented linear feedback shift registers or $\sigma$-LFSRs, and they reduce to classical LFSRs or homogeneous linear recurrences of order $n$ (with coefficients in $\Fq$) when $m=1$. Primitive vector recurrences are of interest in cryptography since they are useful in pseudorandom number generation, or alternatively, for designing fast, secure, and efficient stream ciphers. While the study of (ordinary) LFSRs is classical (see, e.g., \cite[Chap. 8]{LN}), vector recurrences and the corresponding multiple recursive method appears to have been first studied by Niederreiter \cite{N1, N2}. This method seems to have been rediscovered by Zeng, Han and He \cite{Zeng} in the guise of $\sigma$-LFSRs. Prior to that,  
generalizations of LFSRs (that turn out to be special cases of vector recurrences of Niederreiter) were studied by Tsaban and Vishne \cite{TV} and later by Dewar and Panario \cite{DP}, and these are called transformation shift registers or TSRs. We refer to the recent paper of Hasan, Panario and Wang \cite{HPW} for more on TSRs and related developments. 

Enumerating primitive LFSRs of a given order is easy and well-known, whereas it is an open question in the 
case of $\sigma$-LFSRs. 
The following conjectural formula was proposed in \cite{GSM} as a $q$-ary version of a conjecture of 
Zeng, Han and He \cite{Zeng}. 
\medskip

{\bf Primitive Vector Recurrence Conjecture (PVRC):}
The number of primitive vector recurrences of order $n$ over $\Fq^m$ is 
\begin{equation}
\label{NoSigmaLFSR}
\frac{\phi(q^{mn}-1)}{mn} \, q^{m(m-1)(n-1)} \displaystyle \prod_{i=1}^{m-1}(q^m-q^i).
\end{equation}
  
To relate the above to matrices, 
note that the maximum possible order of an element of the finite group $\GL_d(\Fq)$ is $q^d-1$ (see, e.g., \cite[Prop. 3.1]{GSM}) and elements 
of order $q^d-1$ are called \emph{Singer cycles} in $\GL_d(\Fq)$. By an \emph{$(m,n)$-block companion Singer cycle} over $\Fq$
we shall mean a Singer cycle $T$ in $\GL_{mn}(\Fq)$ of the form
\begin{equation}
\label{typeT} 
T =
\begin {pmatrix}
\mathbf{0} & \mathbf{0} & \mathbf{0} & . & . & \mathbf{0} & \mathbf{0} & C_0\\
I_m & \mathbf{0} & \mathbf{0} & . & . & \mathbf{0} & \mathbf{0} & C_1\\
. & . & . & . & . & . & . & .\\
. & . & . & . & . & . & . & .\\
\mathbf{0} & \mathbf{0} & \mathbf{0} & . & . & I_m & \mathbf{0} & C_{n-2}\\
\mathbf{0} & \mathbf{0} & \mathbf{0} & . & . & \mathbf{0} & I_m & C_{n-1}
\end {pmatrix}, 
\end{equation}
where $C_0, C_1, \dots , C_{n-1}\in \M_m(\Fq)$ and $I_m$ denotes the $m\times m$ identity matrix over $\Fq$, while $\mathbf{0}$ indicates the zero matrix in $\M_m(\Fq)$. The relation between $(m,n)$-block companion matrices such as $T$ above and vector recurrences will be clearer if one observes that \eqref{sigmalfsr} is equivalent to the relation $S_{i+1} = S_iT$ for $i\ge 0$, where $S_i$ is the $i^{\rm th}$ state vector $({\mathbf{s}}_i, \dots, {\mathbf{s}}_{i+n-1})$ and $T$ is as in \eqref{typeT}. Primitive vector recurrences correspond precisely to $(m,n)$-block companion Singer cycles \cite[Thm. 5.2]{GSM} and thus the PVRC is equivalent to

\medskip

{\bf Block Companion Singer Cycle Conjecture (BCSCC):}
The number of $(m,n)$-block companion Singer cycles over $\Fq$ is given by \eqref{NoSigmaLFSR}.

\medskip

It turns out that the map that sends a matrix in $M_{mn}(\Fq)$ to its characteristic polynomial maps the set of $(m,n)$-block companion Singer cycles over $\Fq$ onto the set of primitive polynomials in $\Fq[X]$ of degree $mn$ (cf. \cite[Thm. 6.1]{GSM}). Recall that a polynomial in $\Fq[X]$ of degree $d$ is said to be \emph{primitive} if it is the minimal polynomial over $\Fq$ of a generator of the cyclic  group $\FF_{q^d}^*$ of nonzero elements of $\FF_{q^d}$. Evidently, the number of primitive polynomials in $\Fq[X]$ of degree $d$ is $\varphi(q^d-1)/d$. With this in view, BCSCC is implied by the following stronger conjecture.  

\medskip

{\bf Primitive Fiber Conjecture (PFC):} 
For any primitive polynomial $f$ in $\Fq[X]$ of degree $mn$, the number of $(m,n)$-block companion Singer cycles over $\Fq$ having $f$ as its characteristic polynomial is
\begin{equation}
\label{NoFiber}
 q^{m(m-1)(n-1)} \displaystyle \prod_{i=1}^{m-1}(q^m-q^i).
\end{equation}

The above conjecture, proposed first in \cite{GSM}, was further strengthened in \cite{GR} as follows.

\medskip

{\bf Irreducible Fiber Conjecture (IFC):} 
For any irreducible polynomial $f$ in $\Fq[X]$ of degree $mn$, the number of $(m,n)$-block companion Singer cycles over $\Fq$ having $f$ as its characteristic polynomial is
given by \eqref{NoFiber}.

\medskip

To relate irreducible fibers to splitting subspaces, it suffices to observe (see, e.g., Lemmas 5.1 and 5.2 of  \cite{GR}) that if $f\in \Fq[X]$ is an irreducible polynomial of degree $mn$ and if $\alpha\in \Fqmn$ is any root of $f$, then the 
the number of $(m,n)$-block companion Singer cycles over $\Fq$ having $f$ as its characteristic polynomial is precisely 
$N(\alpha, m,n;q)/(q^{mn}-1)$, where $N(\alpha, m,n;q)$ is as in Section \ref{sec2} above. With this in view, the relation between the conjectures stated above and the Splitting Subspace Conjecture (SSC) as well as the Pointed Splitting Subspace Conjecture (PSSC) 
stated earlier in this paper can be summarized as follows. 
$$
\text{PSSC } \Longleftrightarrow \text{ SSC } \Longleftrightarrow \text{ IFC } \Longrightarrow \text{ PFC } \Longrightarrow \text{ BCSCC } \Longleftrightarrow \text{ PVRC.}
$$
In view of the results proved in the previous sections, it is seen that each of these conjectures holds in the affirmative when $m\le 2$ or $n\le 1$. It may also be remarked that in order to prove the PFC, one may fix a primitive polynomial $f\in \Fq[X]$ of degree $mn$ and a matrix $T\in \M_{mn}(\Fq)$ with $f$ as its characteristic polynomial. Then matrices in $M_{mn}(\Fq)$ having $f$ as its characteristic polynomial are necessarily similar to $T$. With this in view, Lachaud \cite{L} has made a fine analysis of the similarity class of $T$ and the collection of $(m,n)$-block companion matrices in it. He shows that $T$ may be chosen to be a block diagonal matrix and the $(m,n)$-block companion matrices in the  similarity class of $T$ correspond to $V^{-1}TV$, where $V$ is a 
so called 
block Vandermonde matrix. This analysis and the results of Lachaud \cite{L} lend further insight into the above conjectures. Nonetheless, the general case remains open. 

\bibliographystyle{amsplain}


\end{document}